\documentclass[12pt]{amsart}
\usepackage{cases}
\usepackage{txfonts}
\usepackage{latexsym}
\usepackage{amsthm}
\usepackage{mathrsfs}
\usepackage{amssymb, amsmath}
\usepackage{enumerate}
\usepackage[symbol]{footmisc}

\theoremstyle{plain}
\newtheorem{theorem}{Theorem}[section]

\newtheorem{lemma}{Lemma}[section]
\newtheorem{proposition}{Proposition}[section]
\newtheorem{Conjecture}{Conjecture}[section]
\theoremstyle{definition}

\theoremstyle{remark}

\numberwithin{equation}{section}

\def\qed{\ifhmode\textqed\fi
   \ifmmode\ifinner\quad\qedsymbol\else\dispqed\fi\fi}
\def\textqed{\unskip\nobreak\penalty50
    \hskip2em\hbox{}\nobreak\hfil\qedsymbol
    \parfillskip=0pt \finalhyphendemerits=0}
\def\dispqed{\rlap{\qquad\qedsymbol}}

\title[]{The $k$-tuple Prime Difference Champion}

\author{Libo Wu}
\address{ School of Mathematics\\
          Hefei University of Technology\\
          Hefei 230009, P.R. China}
\email{wulibohfut@163.com}

\author{Xiaosheng Wu*}
\address{ School of Mathematics\\
          Hefei University of Technology\\
          Hefei 230009, P.R. China}
\email{xswu@amss.ac.cn}
\thanks{*Corresponding author}
\keywords{Hardy-Littlewood Prime $k$-tuple Conjecture; Prime difference; Primorial number; Jumping champion; Singular series.}
\subjclass[2010]{Primary 11N05; Secondary 11P32, 11N36 }
\begin{document}
\begin{abstract}
Let $D_{k}$ be a set with $k$ distinct elements of integers such that $d_{1}<d_{2}<\cdots<d_{k}$. We say $D_{k}^{*}$ is a $k$-tuple prime difference champion ($k$-tuple PDC) for primes $\le x$ if the set $D_{k}^{*}$ is the most probable differences among $k+1$ primes up to $x$. Unconditionally we prove that the $k$-tuple PDCs go to infinity and further have asymptotically the same number prime factors when weighted by logarithmic derivative as the porimorials. Assuming an appropriate form of the Hardy-Littlewood Prime $k$-tuple Conjecture, we obtain that the $k$-tuple PDCs are infinite square-free numbers containing any large primorial as factor when $x\rightarrow \infty$.
\end{abstract}

\maketitle

\section{Introduction}
The differences among primes are the most insightful subject of the primes and the study of prime differences has been conducted for a long time. If restricted to consecutive primes, the differences are known as gaps of primes. In the issues of the $1977$-$1978$ volumes of Journal of Recreational Mathematics, Nelson proposed the problem \cite{Nel1}\cite{Nel2}: ``Find the most probable difference between consecutive primes." Soon after, in $1980$, on the truth of the Hardy-Littlewood Prime Pair Conjecture, Erd\"{o}s and Straus showed that there is no most likely difference between consecutive primes \cite{ES}, since they found that the most likely difference grows as the number they considered becomes larger.\par
 In $1993$, Conway invented the term jumping champion to refer to the most common prime gap between consecutive primes not exceeding $x$. Let $p_{n}$ denote the $n$-th prime. The jumping champions are the integers of $d$ for which the counting function
\begin{equation}\label{conwaydef}
N(x,d)=\sum_{\substack{p_{n+1}\leq x\\p_{n+1}-p_{n}=d}}1
\end{equation}
attains its maximum
\begin{equation}
N^{*}(x)=\max_{d} N(x,d).
\end{equation}\par
In $1999$, based on heuristics arguments and extensive numerical studies, Odlyzko, Rubinstein and Wolf \cite{ORW} enunciated the following two significant conjectures for this problem:
\begin{Conjecture}
\label{con1.1}
The jumping champions greater than 1 are 4 and the primorials 2, 6, 30, 210, 2310,$\cdots$.
\end{Conjecture}

\begin{Conjecture}
\label{con1.2}
The jumping champions tend to infinity. Furthermore, any fixed prime $p$ divides all sufficiently large jumping champions.
\end{Conjecture}
Conjecture \ref{con1.2} is a weaker consequence of Conjecture \ref{con1.1}. The first assertion of Conjecture \ref{con1.2} was proved in 1980 by Erd\"{o}s and Straus \cite{ES} under the assumption of the truth of the Hardy-Littlewood Prime Pair Conjecture. In 2011, Goldston and Ledoan \cite{GL1} extended successfully the method in \cite{ES} to give a complete proof of Conjecture \ref{con1.2} under the same assumption. Shortly after that, in \cite{GL2} they also proved Conjecture \ref{con1.1} for sufficiently large jumping champions by requiring information about a strong form of the Hardy-Littlewood Prime Pair Conjecture and the Prime Triple Conjecture.

In 2012, motivated by Goldston and Ledoan's work, Feng and Wu \cite{FW} considered the problem what are the most probable differences among $k+1$ consecutive primes for any given $k\geq1$. With an appropriate form of the Hardy-Littlewood Prime $k$-tuple Conjecture, Conjecture \ref{con1.2} for $k$-tuple jumping champions is proved in \cite{FW}, and some properties of Conjecture \ref{con1.1} are also obtained in \cite{FW}.


It is important to note that, when restricted to differences among consecutive primes, all results about Jumping Champions such as in \cite{FW, GL1, GL2} are based on the Hardy-Littlewood Prime $k$-tuple Conjecture for different $k$ respectively. Actually, we know virtually nothing unconditionally about the Jumping Champions among consecutive primes. We can not even disprove that $2$ is not the Jumping Champions for all sufficiently large $x$ now.

It is natural to consider differences among primes but not consecutive primes. In $2016$, Funkhouser. \emph{etc} \cite{FGSS} considered the most common differences among primes not exceeding $x$, which they named as the Prime Difference Champions (PDCs). They proved that PDCs run through the primorials for sufficiently large $x$ under an appropriate form of the Hardy-Littlewood Prime Pair Conjecture. It is inspiring that some unconditional results were obtained in \cite{FGSS}. Actually, they proved unconditionally that PDCs go to infinity as $x$ runs to infinity, and furthermore, PDCs have asymptotically the same number prime factors when weighted by logarithmic derivative as the primorials.

In this paper, we put our concentration on the problem what are the most probable differences among $k+1$ distinct primes with $k\geq1$. Firstly, we should formally give description of the most probable differences among $k+1$ distinct primes. Let $D _{k}=\{d_{1},\cdots,d_{k}\}$ be a set of $k$ distinct integers with $d_{1}<d_2<\cdots<d_{k}$. For sufficiently large $x$, we say that a set $D_{k}$, denoted by $D_{k}^{*}$, is a $k$-tuple prime difference champion for all primes $\le x$ if it makes the counting function
\begin{equation}\label{ktuplepdcdef}
 G_{k}(x,D_{k})=\sum_{\substack{p\leq x-d_{k}\\p+d_{i}\hspace{0.05cm}are\hspace{0.05cm}all\hspace{0.05cm}primes\\1\leq i\leq k}}1
\end{equation}
attain its maximum
\begin{equation}\label{ktuplepdcdef1}
 G_{k}(x,D_{k}^{*})=\max_{D_{k}} G_{k}(x,D_{k}).
\end{equation}

Actually we obtain that the $k$-tuple PDCs go to infinity and have many prime factors.
\begin{theorem} \label{mainthm1.1}
Let $k$ be any given positive integer and $x$ be a sufficiently large real number. Suppose $D_{k}^{*}=\{d_{1}^*,\cdots,d_{k}^*\}$ is a $k$-tuple PDC for all primes $\le x$ and  $d^{*}$ is the greatest common divisor of all its elements. Then $d^{*}\rightarrow\infty$ as $x\rightarrow\infty$ and the number of distinct prime factors of $d^{*}$ also goes to infinity as $x\rightarrow\infty$.
\end{theorem}

Let\begin{align}
\Delta_{\{0\}\cup D_{k}}=\prod_{0\leq j<i\leq k}\bigg(d_{i}-d_{j}\bigg)
\end{align}
with $d_{0}=0$.

The following theorem says that $d^*$ has few primes not dividing it, and so does $\Delta_{\{0\}\cup D_{k}^*}$.

\begin{theorem}\label{mainthm1.2}
For sufficiently large real number $x$, let $D_{k}^{*}$ be a $k$-tuple PDC for all primes $\le x$ and  $d^{*}$ be the greatest common divisor of all its elements. We have
\begin{equation}\label{importantinequation}
\sum_{{p\nmid d^*}\atop{p\le 2\log x}}\frac1p\le\log k!+k\log2+\log(2(k+1)^2)+o(1)
\end{equation}
and
\begin{align}\label{+1.7}
\sum_{{p\nmid \Delta_{\{0\}\cup D_{k}^*}}\atop{p\le 2\log x}}\frac1p\le\frac1k\log k!+\log2+\frac1k\log(2(k+1)^2)+o(1).
\end{align}
\end{theorem}

The first two items in the right-hand side of \eqref{importantinequation} actually come from the coefficient of the well-known sieve upper bound given in \eqref{2.11}, and the case also arises in \eqref{+1.7}.

With an appropriate form of the Hardy-Littlewood Prime $k$-tuple Conjecture, we also prove that the $gcd$ of a $k$-tuple PDC is a square-free number containing any large primorial as factor when $x\rightarrow\infty$. This is presented formally in the following theorem.

\begin{theorem}\label{mainthm1.3}
Let $x$ be large enough. Suppose $D_{k}^{*}$ is a $k$-tuple PDC for all primes $\le x$ and  $d^{*}$ is the greatest common divisor of all its elements. Assuming the truth of Conjecture \ref{conj2.1}, we have that
 \begin{enumerate}[(i)]
   \item there is a constant $C_k$ depending on $k$ that all primes $\le C_k\log\log x$ divide $d^{*}$;
   \item $d^{*}$ is a square-free number.
 \end{enumerate}
\end{theorem}

Theorem \ref{mainthm1.3} means that all primorials $\le (\log x)^{C_k}$ divide $d^*$. It might be expected that $d^*$ should run through primorials, while further endeavor implies that this seems to be not true. This is due to that the value of singular series is strongly affected by such primes $p$ that only part of elements in $D^*_k$ shares the same residue class modulo $p$.

In the remainder of this paper, $D_{k}$ will always denote a set with $k$ distinct elements of integers such that $d_{1}<d_{2}<\cdots<d_{k}$ and $D_{k}^{*}$ is such a set which is a $k$-tuple PDC for all primes $\le x$ with  $d^{*}$ the greatest common divisor of all its elements. Also, we will denote $D_{k}=d\ast D_{k}^{'}$, where $d$ is the $gcd$ of all elements in $D_{k}$ and $D_{k}^{'}=\{d_{1}^{'},\cdots,d_{k}^{'}\}$ with $d_{i}=dd_{i}^{'}$ for any $i\leq k$.
Throughout the paper the implied constants in $\mathit{O}, \gg, \ll$ and $\mathit{o}$ can depend on $k$. Besides, we let $\epsilon$ always denote an arbitrarily small positive number which may have different values according to the context. At last, two notations of the primorial will be used in alternation as required in the rest of this paper. One is $\lfloor x\rfloor^{\sharp}$ which means the largest primorial no greater than $x$ and the other is $p_{n}^{\sharp}=p_1p_2\cdots p_n$.

\section{The Hardy-Littlewood Prime $k$-tuple Conjecture} \label{titleHLconjecture }
In $1923$, a pioneering paper Partitio Numerioum III [9] was published. In that paper, Hardy and Littlewood created and developed an asymptotic and analytic method in additional number theory about the Goldbach's Conjecture. As mentioned, $D_{k}$ is a set with $k$ distinct integers and let $\pi(x,D_{k})$ denote the number of $n\leq x$ such that $n+d_{1},\cdots,n+d_{k}$ are all primes. They formulated an asymptotic formula for $\pi(x,D_{k})$ as follows.\par
Firstly define
\begin{equation} \label{liingeration}
li_{k}(x)=
\int_{2}^{x}\frac{1}{\log^{k} t}dt
\end{equation}
and
\begin{equation} \label{singularseriesdef}
\mathfrak{S}(D_{k})=\prod_p\bigg(1-\frac{1}{p}\bigg)^{-k}\bigg(1-\frac{{\textit{v}_{D_{k}}(p)}}{p}\bigg),
\end{equation}
where $p$ runs through all primes and $\textit{v}_{D_{k}}(p)$ represents the number of distinct residue classes modulo $p$ occupied by elements of $D_{k}$. If $\mathfrak{S}(D_{k})\neq 0$, then Hardy-Littlewood Prime $k$-tuple Conjecture tells us that
\begin{equation}\label{upperbound}
\pi(x,D_{k})\sim\mathfrak{S}(D_{k})\hspace{0.1cm}li_{k}(x),\hspace{0.15cm}as\hspace{0.15cm} x\rightarrow\infty.
\end{equation}\par
If $\textit{v}_{D_{k}}({p})=p$ for some primes, then $\mathfrak{S}(D_{k})=0$, in which case $\pi(x,D_{k})$ is bounded by $k$. It is sure that this case can not be a PDC, so we will ignore this case in the remainder of the paper without influence.

One may note  that the counting function $G_{k}(x,D_{k})$ we defined is somewhat different from $\pi(x,D_{k})$. By comparing their definitions, we can find the fact that
 \begin{equation}\label{Gk}
 G_{k}(x,D_{k})=\pi(x-d_{k},\{0\}\cup D_{k}).
 \end{equation}
 In this case, we have the Hardy-Littlewood Prime $k$-tuple Conjecture for $G_{k}(x,D_{k})$ as
\begin{equation}
G_{k}(x,D_{k})\sim \mathfrak{S}(\{0\}\cup D_{k})li_{k+1}(x-d_{k}),~ \text{as}~x\rightarrow\infty.
\end{equation}
Thus one could state the conjecture as
\begin{equation}
G_{k}(x,D_{k})=\mathfrak{S}(\{0\}\cup D_{k})li_{k+1}(x-d_{k})+E(x,D_{k}),~ \text{as}~ x\rightarrow\infty,
\end{equation}
where $E(x,D_{k})$ represents an error term.

Although  Hardy and Littlewood did not specifically consider the situation where $d=d(x)\rightarrow\infty$ in their original conjecture, it is reasonable to suppose the Hardy-Littlewood Prime $k$-tuple Conjecture will hold in this situation. To obtain our conditional results in Theorem \ref{mainthm1.3}, the following form of the Hardy-Littlewood $k$-tuple Conjecture will be needed.
\begin{Conjecture}\label{conj2.1}
If $\mathfrak{S}(\{0\}\cup D_{k})\neq 0$, then
\begin{equation}\label{HLconjecture1}
G_{k}(x,D_{k})=\mathfrak{S}(\{0\}\cup D_{k})Li_{k+1}(x,D_{k})+E(x,D_{k}),\hspace{0.15cm}as\hspace{0.15cm} x\rightarrow\infty,
\end{equation}
where
\begin{equation}\label{li}
Li_{k+1}(x,D_{k})=
\int_{2}^{x-d_{k}}\frac{1}{\log^{k+1} t}dt
\end{equation}
and the error term
\begin{equation}\label{errorterm}
E(x,D_{k})=\mathit{o}\bigg(\frac{x}{\log^{k+3}x}\bigg)
\end{equation}
holds uniformly for $D_{k}\subset\Big[2,\frac{2^{k+1}(k+1)!}{2^{k+1}(k+1)!+1}x\Big]$.
\end{Conjecture}

A strong form of Conjecture \ref{conj2.1} is that for $2\le d_k\le x-x^\epsilon$
\begin{align}
E(x,D_k)\ll(x-d_k)^{\frac12+\epsilon},
\end{align}
while the Conjecture \ref{conj2.1} is enough for us to obtain the conditional results in Theorem \ref{mainthm1.3}.

We will also need the following well-known sieve upper bound:  for any positive integer $k$ and
sufficiently large real $x$,
\begin{equation}\label{2.11}
\pi(x,D_{k})\leq(2^{k}k!+\epsilon)\mathfrak{S}(D_{k})\frac{x}{\log^{k}x}
\end{equation}
uniformly for all $D_k$ with $\mathfrak{S}(D_{k})\neq 0$, which was given in Halberstam and Richert's excellent monograph \cite{HR}.

\section{Unconditional results}
In this section we prove Theorem \ref{mainthm1.1} and Theorem \ref{mainthm1.2}, and the proofs are based on a lower bound of a sum on some large $G_k(x,D_k)$. We deduce this lower bound from a careful discussion on the $k$-th mean of $\pi(x;q,a)$ over $a$ for any positive integer $k$.

At first, we deduce a recursive relation and a lower bound for the $k$-th mean of $\pi(x;q,a)$ over $a$ in the following lemma.

\begin{lemma}\label{lemma3.1}
Let $1\le q\le x$ and $\pi(x;q,a)$ denote the number of primes $\leq x$ which are congruent to $a$ modulo $q$. For any given positive integer $k\ge2$, we have
\begin{align}\label{3.1}
\sum_{1\leq a\leq q}\pi(x;q,a)^{k}\ge\frac{\pi(x)}{\phi(q)}\sum_{1\leq a\leq q}\pi(x;q,a)^{k-1}-\mathcal{E}_k(x,q),
\end{align}
moreover for any $k\ge1$
\begin{align}\label{3.2}
\sum_{1\leq a\leq q}\pi(x;q,a)^{k}\ge\frac{\pi^k(x)}{\phi^{k-1}(q)}-\mathcal{E}_k(x,q),
\end{align}
where $\mathcal{E}_k(x,q)$  denotes an error term which may have different values according to context but meets
\begin{align}
\mathcal{E}_k(x,q)\ll \frac{\pi^{k-1}(x)}{\phi^{k-1}(q)}\log q
\end{align}
all the time.
\end{lemma}
Only formula (\ref{3.1}) will be used to prove our theorems. However, we present \eqref{3.2} in the lemma since it is interesting itself.
\begin{proof}
For $k,~q,~x$ as above, we firstly define $A_k(x,q)$ as follows,
\begin{equation}
\begin{split}
A_{0}(x,q)=\phi(q)=\sum_{{1\le a\le q}\atop{(a,q)=1}}1
\end{split}
\end{equation}
and
\begin{equation}
\begin{split}
A_{k}(x,q)=\sum_{{1\leq a\leq q}\atop{(a,q)=1}}\pi(x;q,a)^{k}
\end{split}
\end{equation}
for $k\ge1$.

Also we define $B_k(x,q)$ by
\begin{equation}\label{3.4}
\begin{split}
B_{k}(x,q)=A_{k}(x,q)-\frac{\pi(x)}{\phi(q)}A_{k-1}(x,q)
\end{split}
\end{equation}
for $k\ge1$.  Since $\pi(x;q,a)=0$ or $1$ for $(a,q)>1$, we have
\begin{align}\label{+3.6}
\sum_{{1\le a\le q}\atop{(a,q)>1}}\pi(x;q,a)^k\le\sum_{p\mid q}1\ll\log q
\end{align}
for any $k\ge1$. Specially, it is easy to see from \eqref{+3.6} that
\begin{align}
A_1(x,q)=\sum_{1\le a\le q}\pi(x;q,a)-\sum_{{1\le a\le q}\atop{(a,q)>1}}\pi(x;q,a)=\pi(x)+O(\log q).
\end{align}
This means that
\begin{equation}
\begin{split}
B_{1}(x,q)=A_{1}(x,q)-\frac{\pi(x)}{\phi(q)}A_{0}(x,q)=O(\log q).
\end{split}
\end{equation}

For $k\ge2$, note that
\begin{equation}
\begin{split}
\sum_{{1\leq a\leq q}\atop{(a,q)=1}}\pi(x;q,a)^{k-2}\bigg(\pi(x;q,a)-\frac{\pi(x)}{\phi(q)}\bigg)^{2}\geq 0.
\end{split}
\end{equation}
Expanding the parenthesis above we have
\begin{equation}\label{Ak}
\begin{split}
B_{k}(x,q)-\frac{\pi(x)}{\phi(q)}B_{k-1}(x,q)\ge0.\\
\end{split}
\end{equation}
Since $B_1(x,q)=O(\log q)$, we can conclude that
\begin{equation}\label{3.8}
\begin{split}
B_{k}(x,q)\ge0-\mathcal{E}_k(x,q)
\end{split}
\end{equation}
for any $k\ge2$, which means that
\begin{align}
\sum_{{1\leq a\leq q}\atop{(a,q)=1}}\pi(x;q,a)^{k}\ge\frac{\pi(x)}{\phi(q)}\sum_{{1\leq a\leq q}\atop{(a,q)=1}}\pi(x;q,a)^{k-1}-\mathcal{E}_k(x,q).
\end{align}
Then due to \eqref{+3.6} we may drop the condition $(a,q)=1$ in both sides above to obtain \eqref{3.1}.

Employing the definition of $B_k$ in (\ref{3.8}) we have
\begin{equation}
\begin{split}
A_{k}(x,q)-\frac{\pi(x)}{\phi(q)}A_{k-1}(x,q)\ge0-\mathcal{E}_k(x,q)
\end{split}
\end{equation}
for any $k\ge2$.
Since $A_1(x,q)=\pi(x)+O(\log q)\ge\pi(x)-\mathcal{E}_1(x,q)$ and 
\begin{align}
\sum_{i=1}^{k-1}\bigg(\frac{\pi(x)}{\phi(q)}\bigg)^{k-i}\mathcal{E}_i(x,q)=\mathcal{E}_k(x,q), 
\end{align}
we have by induction that
\begin{equation}
\begin{split}
A_{k}(x,q)\ge\frac{\pi^k(x)}{\phi^{k-1}(q)}-\mathcal{E}_k(x,q),
\end{split}
\end{equation}
which means
\begin{equation}\label{PiK}
\begin{split}
\sum_{{1\leq a\leq q}\atop{(a,q)=1}}\pi(x;q,a)^{k}\ge\frac{\pi^k(x)}{\phi^{k-1}(q)}-\mathcal{E}_k(x,q)
\end{split}
\end{equation}
for any $k\ge1$. Then due to \eqref{+3.6} we may drop the condition $(a,q)=1$ in the left-hand side above and obtain \eqref{3.2}.
\end{proof}

We now come to deduce the lower bound for a sum on some large $G_k(x,D_k)$ from the recursive relation provided in Lemma \ref{lemma3.1}.
\begin{lemma}\label{lembelowbound}
Let $1\leq q\leq \frac{x}{\log^{2}x}$. For any given integer $k\geq 1$ we have
\begin{equation}\label{lembelowbound1}
\begin{split}
&\sum_{{1< d_{1}<d_{2}\cdots<d_{k}\leq x}\atop{q\mid d_{i},1\leq i\leq k}}G_{k}(x,\{d_{1},d_{2},\cdots,d_{k}\})\ge\frac1{(k+1)!}\frac{x^{k+1}}{\phi(q)^k\log^{k+1}x}(1+o(1)).
\end{split}
\end{equation}
\end{lemma}
\begin{proof}
The proof is based on formula (\ref{3.1}) in Lemma \ref{lemma3.1}. Note that
\begin{equation}\label{lemma5-1}
\begin{split}
\sum_{1\leq a\leq q}\pi(x;q,a)^{k+1}
=&\sum_{1\leq a\leq q}\bigg(\sum_{\substack{p_{1},\cdots,p_{k+1}\leq x\\p_{1}\equiv \cdots\equiv p_{k+1}\equiv a(\text{mod} q)}}1\bigg)
=\sum_{\substack{p_{1},\cdots,p_{k+1}\leq x\\p_{1}\equiv \cdots\equiv p_{k+1}(\text{mod} q)}}1.
\end{split}
\end{equation}
Let $\mathop{{\sum}'}$ denote the sum over distinct primes. We may classify the last sum above according to the number of distinct primes as follows
\begin{equation}\label{3.15}
\begin{split}
\sum_{1\leq a\leq q}\pi(x;q,a)^{k+1}=&C_{k+1,0}\mathop{{\sum}'}_{\substack{p_{1},\cdots,p_{k+1}\leq x\\p_{1}\equiv \cdots\equiv p_{k+1}(\text{mod} q)}}1+C_{k+1,1}\mathop{{\sum}'}_{\substack{p_{1},\cdots,p_{k}\leq x\\p_{1}\equiv \cdots\equiv p_{k}(\text{mod} q)}}1\\
&+\cdots+C_{k+1,i}\mathop{{\sum}'}_{\substack{p_{1},\cdots,p_{k+1-i}\leq x\\p_{1}\equiv \cdots\equiv p_{k+1-i}(\text{mod} q)}}1+\cdots+C_{k+1,k}\pi(x),
\end{split}
\end{equation}
where $C_{k+1,i}$ are combination coefficients which are decided only by $k$ and $i$ for $1\le i\le k$. It is easy to see that
\begin{align}
C_{k+1,0}=1,~C_{k+1,1}=\binom{k+1}{2},~C_{k+1,2}=\frac12\binom{k+1}{2}\binom{k-1}{2}+\binom{k+1}{3}, ~C_{k+1,k}=1\notag
\end{align}
and a very rough bound for all $1\le i\le k+1$ that
\begin{align}
C_{k+1,i}\le \binom{k+1}{i}\frac{(k+1-i)^i}{i},
\end{align}
which is enough for our proof.
Similarly we have
\begin{equation}\label{3.18}
\begin{split}
\sum_{1\leq a\leq q}\pi(x;q,a)^{k}=&\mathop{{\sum}'}_{\substack{p_{1},\cdots,p_{k}\leq x\\p_{1}\equiv \cdots\equiv p_{k}(\text{mod} q)}}1+C_{k,1}\mathop{{\sum}'}_{\substack{p_{1},\cdots,p_{k-1}\leq x\\p_{1}\equiv \cdots\equiv p_{k-1}(\text{mod} q)}}1\\
&+\cdots+C_{k,i}\mathop{{\sum}'}_{\substack{p_{1},\cdots,p_{k-i}\leq x\\p_{1}\equiv \cdots\equiv p_{k-i}(\text{mod} q)}}1+\cdots+\pi(x).
\end{split}
\end{equation}
Replace $k$ by $k+1$ in \eqref{3.1} and employ (\ref{3.15}), (\ref{3.18}) in it to have
\begin{align}\label{+1}
\mathop{{\sum}'}_{\substack{p_{1},\cdots,p_{k+1}\leq x\\p_{1}\equiv \cdots\equiv p_{k+1}(\text{mod} q)}}1\ge&\bigg(\frac{\pi(x)}{\phi(q)}-C_{k+1,1}\bigg)\mathop{{\sum}'}_{\substack{p_{1},\cdots,p_{k}\leq x\\p_{1}\equiv \cdots\equiv p_{k}(\text{mod} q)}}1+\cdots\notag\\
+\bigg(\frac{\pi(x)}{\phi(q)}&C_{k,i}-C_{k+1,i+1}\bigg)\mathop{{\sum}'}_{\substack{p_{1},\cdots,p_{k-i}\leq x\\p_{1}\equiv \cdots\equiv p_{k-i}(\text{mod} q)}}1+\cdots+\bigg(\frac{\pi(x)}{\phi(q)}-1\bigg)\pi(x)+O\bigg(\frac{\pi^{k}(x)}{\phi^{k}(q)}\log q\bigg).
\end{align}
Note that $\pi(x)/\phi(q)\rightarrow\infty$ as  $x\rightarrow\infty$ and
\begin{align}
\mathop{{\sum}'}_{\substack{p_{1},\cdots,p_{k-i}\leq x\\p_{1}\equiv \cdots\equiv p_{k-i}(\text{mod} q)}}1\ge0
\end{align}
for any $1\le i\le k-1$. Thus we may discard all terms behind the first term in the right-hand side of \eqref{+1} to have
\begin{align}\label{3.22}
\mathop{{\sum}'}_{\substack{p_{1},\cdots,p_{k+1}\leq x\\p_{1}\equiv \cdots\equiv p_{k+1}(\text{mod} q)}}1\ge&\bigg(\frac{\pi(x)}{\phi(q)}-C_{k+1,1}\bigg)\mathop{{\sum}'}_{\substack{p_{1},\cdots,p_{k}\leq x\\p_{1}\equiv \cdots\equiv p_{k}(\text{mod} q)}}1+O\bigg(\frac{\pi^{k}(x)}{\phi^{k}(q)}\log q\bigg).
\end{align}
Since for any $k\ge2$
\begin{align}
\mathop{{\sum}'}_{\substack{p_{1},\cdots,p_{k}\leq x\\p_{1}\equiv\cdots\equiv p_{k}(\text{mod} q)}}1&=k!\sum_{\substack{p_{1}<\cdots< p_{k}\leq x\\p_{1}\equiv \cdots\equiv p_{k}(\text{mod} q)}}1\notag\\
&=k!\sum_{{1< d_1<\cdots< d_{k-1}\leq x}\atop{q\mid d_{i},1\leq i\leq k-1}}\sum_{{p_{1},\cdots,p_{k}\leq x}\atop{p_{i+1}-p_{1}=d_{i}}}1\notag\\
&=k!\sum_{{1< d_1<\cdots< d_{k-1}\leq x}\atop{q\mid d_{i},1\leq i\leq k-1}}G_{k-1}(x,\{d_{1},d_{2},\cdots,d_{k-1}\}),
\end{align}
we have from (\ref{3.22}) that
\begin{align}\label{3.24}
\sum_{{1< d_1<\cdots< d_{k}\leq x}\atop{q\mid d_{i},1\leq i\leq k}}G_{k}&(x,\{d_{1},d_{2},\cdots,d_{k}\})\ge\notag\\
&\frac1{k+1}\bigg(\frac{\pi(x)}{\phi(q)}-\frac{k(k+1)}2\bigg)\sum_{{1< d_1<\cdots< d_{k-1}\leq x}\atop{q\mid d_{i},1\leq i\leq k-1}}G_{k-1}(x,\{d_{1},d_{2},\cdots,d_{k-1}\})+O\bigg(\frac{\pi^{k}(x)}{\phi^{k}(q)}\log q\bigg).
\end{align}
For $k=1$, it is proved in \cite{FGSS} that
\begin{align}
\sum_{{1< d_1\leq x}\atop{q\mid d_{1}}}G_{1}(x,\{d_{1}\})\ge\frac12\bigg(\frac{x^2}{\phi(q)\log^2x}-\frac{x}{\log x}\bigg)(1+o(1)).
\end{align}
Employing this into (\ref{3.24}) and by induction, we have
\begin{align}
\sum_{{1< d_1<\cdots< d_{k}\leq x}\atop{q\mid d_{i},1\leq i\leq k}}G_{k}&(x,\{d_{1},d_{2},\cdots,d_{k}\})\ge\frac1{(k+1)!}\frac{x^{k+1}}{\phi(q)^k\log^{k+1}x}(1+o(1)).
\end{align}
Thus we prove the lemma.
\end{proof}

We also need following familiar results deduced from Merten's formula.
\begin{lemma}\label{mertens}(Merten).
Let $0<a<b$ be any given constants. We have
\begin{equation}\label{mertens1}
\prod_{\substack{p\leq y}}\bigg(1-\frac{1}{p}\bigg)^{-1}=e^{\gamma}{\log y}+\mathit{O}(1)
\end{equation}
and
\begin{align}
\sum_{ay\le p\leq by}\frac{1}{p}\ll\frac{1}{\log y},
\end{align}
where $\gamma$ is Euler's constant and the constant in $\ll$ is decided by $a,b$.
\end{lemma}

\noindent
\emph{Proof of the Theorem \ref{mainthm1.1}. }Since $D_k^*$ is a $k$-tuple PDC, it is obvious from Lemma \ref{lembelowbound} that
\begin{align}\label{3.29}
G_{k}(x,D_k^*)&\ge k!\bigg(\frac qx\bigg)^k\sum_{{1< d_1<\cdots< d_{k}\leq x}\atop{q\mid d_{i},1\leq i\leq k}}G_{k}(x,\{d_{1},d_{2},\cdots,d_{k}\})\notag\\
&\ge\frac1{k+1}\bigg(\frac{q}{\phi(q)}\bigg)^k\frac{x}{\log^{k+1}x}(1+o(1)).
\end{align}
By this and the trivial upper bound
\begin{align}
G_{k}(x,D_k^*)\le x-d_k^*,
\end{align}
we can easily see that $x-d_k^*\rightarrow\infty $ as $x\rightarrow \infty$. Thus we may use the well-known sieve upper bound (2.11) and the formula (\ref{Gk}) to have that
\begin{align}\label{3.30}
G_{k}(x,D_k^*)&\leq (2^{k+1}{(k+1)!}+\epsilon)\mathfrak{S}(\{0\}\cup D_{k}^{*})\frac{x-d^{*}_{k}}{\log^{k+1} (x-d^*_k)}\notag\\
&\leq (2^{k+1}{(k+1)!}+\epsilon)\mathfrak{S}(\{0\}\cup D_{k}^{*})\frac{x}{\log^{k+1} x}.
\end{align}
Then we conclude from (\ref{3.29}) and (\ref{3.30}) that
\begin{align}
\mathfrak{S}(\{0\}\cup D_{k}^{*})\ge\frac1{2^{k+1}k!(k+1)^2}\bigg(\frac{q}{\phi(q)}\bigg)^k(1+o(1)).
\end{align}
Now we choose $q=\lfloor\frac{x}{\log^{2}x}\rfloor^{\sharp}=p_{l}^{\sharp}$. From the Prime Number Theorem we can easily see $p_{l}\sim\log x$. Since $\phi(q)=q\prod_{p\mid q}(1-\frac1p)$, we have
\begin{align}\label{3.32}
\mathfrak{S}(\{0\}\cup D_{k}^{*})\ge\frac1{2^{k+1}k!(k+1)^2}\prod_{p\le p_l}\bigg(1-\frac1p\bigg)^{-k}(1+o(1)).
\end{align}

On the other hand, from the definition of the singular series in \eqref{singularseriesdef}
\begin{align}\label{kseries}
\mathfrak{S}(\{0\}\cup D_{k}^*)&=\prod_p\bigg(1-\frac{1}{p}\bigg)^{-(k+1)}\bigg(1-\frac{{\textit{v}_{\{0\}\cup D_{k}^*}(p)}}{p}\bigg)\notag\\
&=\prod_{p\mid \Delta_{\{0\}\cup D_{k}^*}}\bigg(1-\frac{1}{p}\bigg)^{-(k+1)}\bigg(1-\frac{{\textit{v}_{\{0\}\cup D_{k}^*}(p)}}{p}\bigg)\prod_{p\nmid \Delta_{\{0\}\cup D_{k}^*}}\bigg(1-\frac{1}{p}\bigg)^{-(k+1)}\bigg(1-\frac{k+1}{p}\bigg)\notag\\
&\le\prod_{p\mid d^*}\bigg(1-\frac{1}{p}\bigg)^{-k}\prod_{{p\mid \Delta_{\{0\}\cup D_{k}^*}}\atop{p\nmid d^*}}\bigg(1-\frac{1}{p}\bigg)^{-(k+1)}\bigg(1-\frac{{\textit{v}_{\{0\}\cup D_{k}^*}(p)}}{p}\bigg),
\end{align}
where the inequality $(1-x)^{-m}(1-mx)\leq 1$ for $mx<1$ is used to eliminate the last product in the second line since every $p\le k+1$ has $p\mid \Delta_{\{0\}\cup D_{k}^*}$ for $\mathfrak{S}(\{0\}\cup D_{k}^*)\neq0$. Note that ${\textit{v}_{\{0\}\cup D_{k}^*}(p)}\ge2$ for $p\nmid d^*$. Thus
\begin{align}\label{3.35}
\mathfrak{S}(\{0\}\cup D_{k}^*)&\le\prod_{p\mid d^*}\bigg(1-\frac{1}{p}\bigg)^{-k}\prod_{{p\mid \Delta_{\{0\}\cup D_{k}^*}}\atop{p\nmid d^*}}\bigg(1-\frac{1}{p}\bigg)^{-(k+1)}\bigg(1-\frac{2}{p}\bigg)\notag\\
&\le\prod_{p\mid d^*}\bigg(1-\frac{1}{p}\bigg)^{-k}\prod_{{p\mid \Delta_{\{0\}\cup D_{k}^*}}\atop{p\nmid d^*}}\bigg(1-\frac{1}{p}\bigg)^{-(k-1)}.
\end{align}
Combining (\ref{3.32}) and (\ref{3.35}) we have
\begin{align}\label{3.36}
\prod_{p\mid d^*}\bigg(1-\frac{1}{p}\bigg)^{-k}\prod_{{p\mid \Delta_{\{0\}\cup D_{k}^*}}\atop{p\nmid d^*}}\bigg(1-\frac{1}{p}\bigg)^{-(k-1)}\ge\frac1{2^{k+1}k!(k+1)^2}\prod_{p\le p_l}\bigg(1-\frac1p\bigg)^{-k}(1+o(1)).
\end{align}
This gives
\begin{align}\label{3.37}
\prod_{p\mid d^*}\bigg(1-\frac{1}{p}\bigg)^{-k}\prod_{{p\mid \Delta_{\{0\}\cup D_{k}^*}}\atop{p\nmid d^*}}\bigg(1-\frac{1}{p}\bigg)^{-(k-1)}\ge\frac{e^{k\gamma}}{2^{k+1}k!(k+1)^2}(\log\log x)^{k}(1+o(1))
\end{align}
by Lemma \ref{mertens}. Let $\Omega(n)$ denote the total number of prime divisors of positive integer $n$. It is well known that
\begin{align}
\Omega(n)\le(1+\epsilon)\log n/\log\log n.
\end{align}
Since $\Delta_{\{0\}\cup D_{k}^*}\le x^{k(k+1)/2}$,
we have
\begin{align}
\prod_{p\mid \Delta_{\{0\}\cup D_{k}^*}}\bigg(1-\frac{1}{p}\bigg)^{-1}&
\le\prod_{p\le p_{\Omega(\Delta_{\{0\}\cup D^*_k})}}\bigg(1-\frac{1}{p}\bigg)^{-1}\notag\\
&\le\prod_{p\le k(k+1)\log x}\bigg(1-\frac{1}{p}\bigg)^{-1}\le e^\gamma\log\log x(1+o(1)),
\end{align}
which means
\begin{align}\label{3.39}
\prod_{p\mid d^*}\bigg(1-\frac{1}{p}\bigg)^{-1}\prod_{{p\mid \Delta_{\{0\}\cup D_{k}^*}}\atop{p\nmid d^*}}\bigg(1-\frac{1}{p}\bigg)^{-1}\le e^\gamma\log\log x(1+o(1)).
\end{align}
Then we conclude from (\ref{3.37}) and the $(k-1)$-th power of (\ref{3.39}) that
\begin{align}
\prod_{p\mid d^*}\bigg(1-\frac{1}{p}\bigg)^{-1}\ge\frac{e^{\gamma}}{2^{k+1}k!(k+1)^2}\log\log x(1+o(1)).
\end{align}
Also
\begin{align}
\prod_{p\mid d^*}\bigg(1-\frac{1}{p}\bigg)^{-1}\le\prod_{p\mid \lfloor x\rfloor^{\sharp}}\bigg(1-\frac{1}{p}\bigg)^{-1}=e^{\gamma}\log\log x(1+o(1)).
\end{align}
Taking logarithm in the above two formulae we have
\begin{equation}\label{thmresult}
\sum_{p\mid d^{*}}\frac{1}{p}=\log\log\log x+O(1).
\end{equation}

Therefore $d^{*}\rightarrow\infty$ and the number of distinct prime factors of $d^{*}$ also goes to infinity as $x \rightarrow\infty$, which proves Theorem \ref{mainthm1.1}.\\

\noindent
\emph{Proof of Theorem \ref{mainthm1.2}}.
Deducing from (\ref{3.36}) we have
\begin{align}\label{3.42}
\prod_{p\mid d^*}\bigg(1-\frac{1}{p}\bigg)^{-k}\prod_{{p\mid \Delta_{\{0\}\cup D_{k}^*}}\atop{p\nmid d^*}}\bigg(1-\frac{1}{p}\bigg)^{-(k-1)}\prod_{p\le p_l}\bigg(1-\frac1p\bigg)^{k}\ge\frac1{2^{k+1}k!(k+1)^2}(1+o(1)).
\end{align}
The left-hand side of this equation is not larger than
\begin{align}
\prod_{{p\mid d^*}\atop{p>p_l}}\bigg(1-\frac{1}{p}\bigg)^{-k}\prod_{{p\nmid d^*}\atop{p\le p_l}}\bigg(1-\frac1p\bigg).
\end{align}
By Lemma \ref{mertens} we have the first product
\begin{align}\label{3.44}
1\le\prod_{{p\mid d^*}\atop{p>p_l}}\bigg(1-\frac{1}{p}\bigg)^{-k}&\le\prod_{p_l<p\le p_{l+\Omega(d^*)}}\bigg(1-\frac{1}{p}\bigg)^{-k}\notag\\
&\le\prod_{\frac12\log x\le p\le 2\log x}\bigg(1-\frac{1}{p}\bigg)^{-k}=1+O\bigg(\frac1{\log\log x}\bigg).
\end{align}
Thus we conclude that
\begin{align}
\prod_{{p\nmid d^*}\atop{p\le 2\log x}}\bigg(1-\frac1p\bigg)\ge\frac1{2^{k+1}k!(k+1)^2}(1+o(1))
\end{align}
since we have
\begin{align}
\prod_{p_l\le p\le2\log x}\bigg(1-\frac1p\bigg)=1+O\bigg(\frac1{\log\log x}\bigg)
\end{align}
as \eqref{3.44}.
Taking logarithm and using the fact $x\le-\log(1-x)$ for $0<x<1$ we have
\begin{align}\label{3.46}
\sum_{{p\nmid d^*}\atop{p\le 2\log x}}\frac1p\le\log k!+k\log2+\log(2(k+1)^2)+o(1).
\end{align}

Also, we have the left-hand side of the equation \eqref{3.42} is not larger than
\begin{align}
\prod_{{p\mid \Delta_{\{0\}\cup D_{k}^*}}\atop{p>p_l}}\bigg(1-\frac{1}{p}\bigg)^{-k}\prod_{{p\nmid \Delta_{\{0\}\cup D_{k}^*}}\atop{p\le p_l}}\bigg(1-\frac1p\bigg)^k.
\end{align}
Similarly as (\ref{3.44}) we have
\begin{align}
\prod_{{p\mid \Delta_{\{0\}\cup D_{k}^*}}\atop{p>p_l}}\bigg(1-\frac{1}{p}\bigg)^{-k}=1+O\bigg(\frac1{\log\log x}\bigg),
\end{align}
and so
\begin{align}
\prod_{{p\nmid \Delta_{\{0\}\cup D_{k}^*}}\atop{p\le p_l}}\bigg(1-\frac1p\bigg)^k\ge\frac1{2^{k+1}k!(k+1)^2}(1+o(1)).
\end{align}
Then it follows the same as (\ref{3.46}) that
\begin{align}
\sum_{{p\nmid \Delta_{\{0\}\cup D_{k}^*}}\atop{p\le 2\log x}}\frac1p\le\frac1k\log k!+\log2+\frac1k\log(2(k+1)^2)+o(1).
\end{align}
Thus we complete the proof of Theorem \ref{mainthm1.2}.

\section{Results under the Hardy-Littlewood Prime $k$-tuple Conjecture}

The proof of Theorem \ref{mainthm1.3} is based on Conjecture \ref{conj2.1}. Some preparation deduced from the conjecture will be presented firstly. It mainly contains two aspects: an asymptotic formula for $G_{k}(x,D_{k})$ and a superior set $\widetilde D_{k}$ which makes $G_{k}(x,D_{k})$ large. Then we will deduce a rough range for $D_k^*$. At last we will give the proof of the theorem.

The asymptotic formula is fundamental to the proof. We present it in the following lemma.

\begin{lemma}\label{lem2}
For given integer $k\geq 1$, assume Conjecture \ref{conj2.1}. Let $D_{k}$ be a set of $k$ distinct integers such that $d_{1}<d_{2}<\cdots<d_{k}$ with $\mathfrak{S}(\{0\}\cup D_{k})\neq 0$.
If $D_{k}\subset \Big[2,\frac{2^{k+1}(k+1)!}{2^{k+1}(k+1)!+1}x\Big]$, then
\begin{equation}\label{lem2-1}
\begin{split}
&G_{k}(x,D_{k})=\mathfrak{S}(\{0\}\cup D_{k})\frac{x-d_{k}}{\log^{k+1}x}H(x,D_{k})\bigg(1+\mathit{o}\bigg(\frac{1}{\log^{2}x}\bigg)\bigg),\\
\end{split}
\end{equation}
where
\begin{equation*}
H(x,D_{k})=1+\frac{k+1}{\log x}+\frac{(k+1)(k+2)}{\log^{2} x}+\mathit{O}\bigg(\frac{d_{k}}{x\log x}\bigg)+\mathit{O}\bigg(\frac{1}{\log^{3} x}\bigg).
\end{equation*}
\end{lemma}
\begin{proof}
In view of the error term in Conjecture \ref{conj2.1}, we take
\begin{equation}
Li_{k+1}(x,D_{k})=
\int_{\frac{x}{\log^{k+4}x}}^{x-d_{k}}\frac{1}{\log^{k+1} t}dt+\mathit{O}\bigg(\frac{x}{\log^{k+4}x}\bigg)
\end{equation}
and integrate by parts three times to obtain
\begin{equation}\label{4.3}
\begin{split}
Li_{k+1}(x,D_{k})
=&\frac{x-d_{k}}{\log^{k+1}(x-d_{k})}+(k+1)\frac{x-d_{k}}{\log^{k+2}(x-d_{k})}\\
&+(k+1)(k+2)\frac{x-d_{k}}{\log^{k+3}(x-d_{k})}+O(Li_{k+4}(x,D_{k}))+\mathit{O}\bigg(\frac{x}{\log^{k+4}x}\bigg).\\
\end{split}
\end{equation}
Since $2\leq d_{k}\leq \frac{2^{k+1}(k+1)!}{2^{k+1}(k+1)!+1}x$, we have
\begin{equation}\label{4.4}
\log(x-d_{k})=\log x+\log\bigg(1-\frac{d_{k}}{x}\bigg)=\log x+\mathit{O}\bigg(\frac{d_{k}}{x}\bigg).
\end{equation}
Furthermore, we have the trivial estimate
\begin{align}\label{4.5}
Li_{k+4}(x,D_{k})\ll\frac{x-d_{k}}{\log^{k+4} x}.
\end{align}
Then employ (\ref{4.4}) and (\ref{4.5}) into (\ref{4.3}) to have
\begin{equation}
\begin{split}
Li_{k+1}(x,D_{k})=\frac{x-d_{k}}{\log^{k+1}x}\bigg(1+\frac{k+1}{\log x}+\frac{(k+1)(k+2)}{\log^{2}x}+\mathit{O}\bigg(\frac{d_{k}}{x\log x}\bigg)+\mathit{O}\bigg(\frac{1}{\log^{3} x}\bigg)\bigg).
\end{split}
\end{equation}
Thus \eqref{lem2-1} follows from this and Conjecture \ref{conj2.1}.
\end{proof}

We now come to the superior set $\widetilde D_{k}$. The asymptotic formula in Lemma \ref{lem2} tells us that $H(x,D_{k})$ is essentially constant for sufficiently large $x$. Thus $G_{k}(x,D_{k})$ actually depends on values of $d_k$ and the singular series.  From the definition of the singular series in \eqref{singularseriesdef}, we have
 \begin{equation}
 \mathfrak{S}(\{0\}\cup D_{k})=\prod_{p}\bigg(1-\frac{1}{p}\bigg)^{-(k+1)}\prod_{p\mid d}\bigg(1-\frac{1}{p}\bigg)\prod_{p\nmid d}\bigg(1-\frac{{v_{\{0\}\cup D_{k}}(p)}}{p}\bigg),
 \end{equation}
where $d$ is the great common divisor of all elements in $D_k$. Since $v_{\{0\}\cup D_{k}}(p)\ge2$ for $p\nmid d$, the singular series will enlarge when the number of distinct prime factors in $d$ grows and values of prime factors in $d$ diminish. Hence it seems that the primorial is a good potential choice for $d$. Thus we take
\begin{align}
\widetilde D_{k}=\Big\lfloor{\frac{x}{\log x}}\Big\rfloor^{\sharp}*\mathcal{K}
\end{align}
with $\mathcal{K}=\{1,2,\cdots,k\}$.
If denote $p_{n}^{\sharp}=\lfloor{\frac{x}{\log x}}\rfloor^{\sharp}$, we have $p_n\sim \log x$. One great benefit of this $\widetilde D_{k}$ is owning a singular series with large value, which is indicated in the following lemma.
\begin{lemma}\label{lem4.2}
 Let $D_{k}=\{d_{1},\cdots,d_{k}\}$ be any set of $k$ distinct integers with $d_{1}<d_{2}< \cdots< d_{k}\ll x$. Assume the set $\widetilde D_{k}=\lfloor\frac{x}{\log x}\rfloor^{\sharp}*\mathcal{K}$, then we always have
 \begin{equation}\label{Vinequation}
\frac{\mathfrak{S}(\{0\}\cup D_{k})}{\mathfrak{S}(\{0\}\cup\widetilde{D_{k}})}\leq 1+\frac{C}{\log\log x},
\end{equation}
where $C$ is a constant decided by $k$.
\end{lemma}
\begin{proof}
Note from the definition of the singular series that
\begin{align}
\mathfrak{S}(\{0\}\cup D_{k})\le\prod_p\bigg(1-\frac1p\bigg)^{-(k+1)}\prod_{p\mid \lfloor{\Delta_{\{0\}\cup D_k}}\rfloor^{\sharp}}\bigg(1-\frac1p\bigg)\prod_{p\nmid \lfloor{\Delta_{\{0\}\cup D_k}}\rfloor^{\sharp}}\bigg(1-\frac {k+1}p\bigg)
\end{align}
and
\begin{align}
\mathfrak{S}(\{0\}\cup \widetilde D_{k})=\prod_p\bigg(1-\frac1p\bigg)^{-(k+1)}\prod_{p\mid \lfloor{\frac x{\log x}}\rfloor^{\sharp}}\bigg(1-\frac1p\bigg)\prod_{p\nmid \lfloor{\frac x{\log x}}\rfloor^{\sharp}}\bigg(1-\frac {k+1}p\bigg).
\end{align}
Thus
\begin{align}
\frac{\mathfrak{S}(\{0\}\cup D_{k})}{\mathfrak{S}(\{0\}\cup\widetilde{D_{k}})}&\le\prod_{{p\mid \lfloor{\Delta_{\{0\}\cup D_k}}\rfloor^{\sharp}}\atop{p\nmid \lfloor{\frac x{\log x}}\rfloor^{\sharp}}}\frac{p-1}{p-k-1}\notag\\
&\le\prod_{p_{\Omega(\lfloor{\frac x{\log x}}\rfloor^{\sharp})}< p\le p_{\Omega(\lfloor{\Delta_{\{0\}\cup D_k}}\rfloor^{\sharp})}}\frac{p-1}{p-k-1}\notag\\
&\le\prod_{\frac12\log x\le p\le k(k+1)\log x}\frac{p-1}{p-k-1}.
\end{align}
Then we have by Lemma \ref{mertens}
\begin{align}
\frac{\mathfrak{S}(\{0\}\cup D_{k})}{\mathfrak{S}(\{0\}\cup\widetilde{D_{k}})}\le1+\frac C{\log\log x}.
\end{align}
Thus we complete the proof.
\end{proof}

Denote $\widetilde D_{k}$ by $\{\widetilde d_1, \widetilde d_2,\cdots,\widetilde d_k\}$. Since $\widetilde d_k\le \frac{kx}{\log x}$,
then the asymptotic formula in Lemma \ref{lem2} for $G_{k}(x,\widetilde D_{k})$ is
\begin{align}
G_{k}(x,\widetilde D_{k})=\mathfrak{S}(\{0\}\cup \widetilde D_{k})\frac{x-\widetilde d_{k}}{\log^{k+1}x}H(x,\widetilde D_{k})\bigg(1+\mathit{o}\bigg(\frac{1}{\log^{2}x}\bigg)\bigg)
\end{align}
with
\begin{equation*}
H(x,\widetilde D_{k})=1+\frac{k+1}{\log x}+\frac{(k+1)(k+2)}{\log^{2} x}+\mathit{O}\bigg(\frac{\widetilde d_{k}}{x\log x}\bigg)+O\bigg(\frac{1}{\log^{3} x}\bigg).
\end{equation*}
This indicates that
\begin{align}\label{4.11}
G_{k}(x,\widetilde D_{k})\ge\mathfrak{S}(\{0\}\cup \widetilde D_{k})\frac{x}{\log^{k+1}x}\bigg(1+\frac1{\log x}+O\bigg(\frac1{\log^2x}\bigg)\bigg).
\end{align}

With this preparation we may deduce a rough range for $D_k^*$ in the following proposition.
\begin{proposition}\label{pro4.1}
For sufficiently large $x$, let $D_{k}^{*}=\{d_1^*,d_2^*,\cdots,d_k^*\}$ be a $k$-tuple PDC for primes $\le x$. We have $(1-\delta)\frac{x}{\log^{2} x}<d_k^*<(C+\epsilon)\frac{x}{\log\log x}$ for any given small constants $\epsilon,\delta>0$, where $C$ is the constant as in \eqref{Vinequation}.
\end{proposition}
\begin{proof}
The proof is devoted to discussing four cases of $d_k$ according to the partition of the range $[1,x]$. Since $D_{k}^{*}$ is a best set to make $G_{k}(x,D_{k})$ large, we eliminate each case by finding another preferable set for every $D_k$ with $d_k$ in the case.

\emph{case 1}. Suppose $x-\frac{x}{\log^{k+1}x}\leq d_{k} \leq x$. We see that
\begin{equation}
\begin{split}
G_{k}(x,D_{k})=\sum_{\substack{p\leq x-d_{k}\\p+d_{i}\hspace{0.05cm}are\hspace{0.05cm}all\hspace{0.05cm}primes\\1\leq i\leq k}}1\le\sum_{p\leq x-d_{k}}1\leq x-d_k\leq \frac{x}{\log^{k+1}x}<G_{k}(x,\widetilde D_{k})
\end{split}
\end{equation}
by \eqref{4.11}.

\emph{case 2}. Suppose $\frac{2^{k+1}(k+1)!}{2^{k+1}(k+1)!+1}x< d_{k}\leq x-\frac{x}{\log^{k+1}x}$. Using the well-known sieve upper bound \eqref{2.11} we have
\begin{equation}
\begin{split}
G_{k}(x,D_{k})
&\leq \bigg(2^{k+1}(k+1)!+\epsilon\bigg)\mathfrak{S}(\{0\}\cup D_{k})\frac{x-d_{k}}{\log^{k+1}(x-d_{k})}\\
&\leq \frac{2^{k+1}(k+1)!+\epsilon}{2^{k+1}(k+1)!+1}\mathfrak{S}(\{0\}\cup D_{k})\frac{x}{\log^{k+1}x-(k+1)\log\log x}\\
&<\frac{2^{k+1}(k+1)!+1}{2^{k+1}(k+1)!+2}\mathfrak{S}(\{0\}\cup D_{k})\frac{x}{\log^{k+1}x}\\
\end{split}
\end{equation}
for sufficiently large $x$. Then by Lemma \ref{lem4.2} and \eqref{4.11} we have
\begin{align}
G_{k}(x,D_{k})&<\frac{2^{k+1}(k+1)!+1}{2^{k+1}(k+1)!+2}\frac{\mathfrak{S}(\{0\}\cup D_{k})}{\mathfrak{S}(\{0\}\cup\widetilde{D_{k}})}\mathfrak{S}(\{0\}\cup \widetilde D_{k})\frac{x}{\log^{k+1}x}\notag\\
&\le\frac{2^{k+1}(k+1)!+1}{2^{k+1}(k+1)!+2}G_{k}(x,\widetilde D_{k})\bigg(1+O\bigg(\frac1{\log\log x}\bigg)\bigg)\notag\\
&<G_{k}(x,\widetilde D_{k})
\end{align}
for sufficiently large $x$.

\emph{case 3}. Suppose $\frac{(C+\epsilon)x}{\log\log x}\le d_k\le\frac{2^{k+1}(k+1)!}{2^{k+1}(k+1)!+1}x$ for some $\epsilon>0$. Then by Lemma \ref{lem2} and Lemma \ref{lem4.2}
\begin{equation}\label{dkinequation}
\begin{split}
\frac{G_{k}(x, D_{k})}{G_{k}(x,\widetilde{D_{k}})}
&\leq \bigg(1+\frac{C}{\log\log x}\bigg)\frac{x-d_{k}}{x-\tilde{d_{k}}}\frac{H(x,D_{k})}{H(x,\widetilde{D_{k}})}
\bigg(1+\mathit{o}\bigg(\frac{1}{\log^{2}x}\bigg)\bigg)\\
&\le1-\frac{\epsilon}{2\log\log x}<1
\end{split}
\end{equation}
for sufficiently large $x$.

\emph{case 4}. Suppose $d_k\le(1-\delta)\frac{x}{\log^{2} x}$ for some small constant $\delta>0$. Due to the Prime Number Theorem, we can find a prime $p'\leq\log x$ with $p'\nmid d$ since $d\le d_k$. For $v_{\{0\}\cup D_{k}^{*}}(p^{\prime})\geq 2$, it is easy to see that
\begin{equation}\label{ratio1}
\begin{split}
\frac{\mathfrak{S}(\{0\}\cup D_{k})}{\mathfrak{S}(\{0\}\cup p^{\prime}\ast D_{k})}
&=\bigg(1-\frac{{v_{\{0\}\cup D_{k}}(p^{\prime})}}{p^{\prime}}\bigg)\bigg(1-\frac{1}{p^{\prime}}\bigg)^{-1}\leq 1-\frac{1}{\log x}.
\end{split}
\end{equation}
As $p^{\prime}d_{k}\leq(1-\delta)\frac{x}{\log x}$ we use Lemma \ref{lem2} to have
\begin{equation}\label{contradiction2}
\begin{split}
\frac{G_{k}(x, D_{k})}{G_{k}(x,p'*D_{k})}=&\frac{\mathfrak{S}(\{0\}\cup D_{k})}{\mathfrak{S}(\{0\}\cup p^{\prime}\ast D_{k})}\frac{x-d_k}{x-p'd_k}\bigg(1+O\bigg(\frac1{\log^2x}\bigg)\bigg)\\
&\leq \bigg(1-\frac{1}{\log x}\bigg)\bigg(1+\frac{1-\delta}{\log x}\bigg)\bigg(1+O\bigg(\frac1{\log^2x}\bigg)\bigg)\\
&\le1-\frac{\delta}{2\log x}<1
\end{split}
\end{equation}
for sufficiently large $x$.

Therefore we conclude that $D_k$ can not be a $k$-tuple PDC with $d_k$ in these four ranges, which proves the proposition.
\end{proof}

\noindent
\emph{Proof of (i) in Theorem \ref{mainthm1.3}}.
Since $d_k^*\le(C+\epsilon)\frac{x}{\log\log x}$ by Proposition \ref{pro4.1}, we have from the asymptotic formula in Lemma \ref{lem2} that

\begin{align}
G_{k}(x,\lfloor x^{\frac{1}{2}}\rfloor^{\sharp}\ast\mathcal{K})&\leq G_{k}(x,D_k^*)\le\mathfrak{S}(\{0\}\cup D_{k}^{*})\frac{x}{\log^{k+1}x}\bigg(1+\frac{C+\epsilon}{\log\log x}\bigg)
\end{align}
with
\begin{align}\label{dualequation}
G_{k}(x,\lfloor x^{\frac{1}{2}}\rfloor^{\sharp}\ast\mathcal{K})=\mathfrak{S}(\{0\}\cup\lfloor x^{\frac{1}{2}}\rfloor^{\sharp}\ast\mathcal{K})\frac{x}{\log^{k+1}x}\bigg(1+O\bigg(\frac1{\log x}\bigg)\bigg),
\end{align}
which means
\begin{equation}\label{seriesbelow}
\frac{\mathfrak{S}(\{0\}\cup\lfloor x^{\frac{1}{2}}\rfloor^{\sharp}\ast\mathcal{K})}{\mathfrak{S}(\{0\}\cup D_{k}^{*})}\leq 1+\frac{C+\epsilon}{\log\log x}.
\end{equation}
Let $p'$ be any prime with $p'\ll\log\log x$ but $p'\nmid d^{*}$. It is easy to see that $2\le v_{{\{0\}\cup D_{k}^{*}}}(p')<p'$ as $D_{k}^*$ is a $k$-tuple PDC. Then we have by Lemma \ref{lem4.2}
\begin{align}\label{set2}
\mathfrak{S}(\{0\}\cup D_{k}^{*})\bigg(1+\frac{v_{{\{0\}\cup D_{k}^{*}}}(p')-1}{p'-v_{{\{0\}\cup D_{k}^{*}}}(p')}\bigg)
&=\mathfrak{S}(\{0\}\cup p'\ast D_{k}^{*})\notag\\
&\leq\mathfrak{S}(\{0\}\cup\widetilde{D_{k}})\bigg(1+\frac{C+\epsilon}{\log\log x}\bigg).
\end{align}
Thus combine \eqref{seriesbelow} and \eqref{set2} to have
\begin{equation}\label{factor1}
\begin{split}
1+\frac{v_{{\{0\}\cup D_{k}^{*}}}(p')-1}{p'-v_{{\{0\}\cup D_{k}^{*}}}(p')}\leq\frac{\mathfrak{S}(\{0\}\cup\widetilde{D_{k}})}{\mathfrak{S}(\{0\}\cup\lfloor x^{\frac{1}{2}}\rfloor^{\sharp}\ast\mathcal{K})}\bigg(1+\frac{2C+\epsilon}{\log\log x}\bigg).\\
\end{split}
\end{equation}
A deduction the same as the proof of Lemma \ref{lem4.2} gives
\begin{align}\label{ratioseries1}
\frac{\mathfrak{S}(\{0\}\cup\widetilde{D_{k}})}{\mathfrak{S}(\{0\}\cup\lfloor x^{\frac{1}{2}}\rfloor^{\sharp}\ast\mathcal{K})}
\leq 1+\frac{C'}{\log\log x}
\end{align}
with $C'$ is a constant decided by $k$.
Thus we conclude that
\begin{equation}\label{factor2}
1+\frac{v_{{\{0\}\cup D_{k}^{*}}}(p')-1}{p'-v_{{\{0\}\cup D_{k}^{*}}}(p')}\leq 1+\frac{2C+C'+\epsilon}{\log\log x}
\end{equation}
with $2\leq v_{{\{0\}\cup D_{k}^{*}}}(p')\leq\min(k+1,p'-1)$. This means that $p'> C_k\log\log x$ as $x\rightarrow\infty$, where $C_k$ is a constant decided by $k$. Therefore, when $x$ is large enough, all prime $p'\le C_k\log\log x$ must divide every element of any $k$-tuple PDCs for primes $\le x$. This proves part (i) of Theorem \ref{mainthm1.3}.\\

\noindent
\emph{Proof of (ii) in Theorem \ref{mainthm1.3}}.
The proof is due to contradiction. Let $p'$ be a prime with $p'^2\mid d^{*}$. Then by the asymptotic formula in Lemma \ref{lem2} we have
\begin{equation}\label{contradiction3}
\begin{split}
\frac{G_{k}(x,D_{k}^{*})}{G_{k}\Big(x,\frac1{p'}*D_{k}^{*}\Big)}
&=\frac{\mathfrak{S}(\{0\}\cup D_{k}^{*})}{\mathfrak{S}\Big(\{0\}\cup \frac1{p'}*D_{k}^{*}\Big)}\frac{x-d^{*}_{k}}{x-\frac{d^*_k}{p'}}\bigg(1+o\bigg(\frac{d^*_k}{x}\bigg)+\mathit{o}\bigg(\frac{1}{\log^{2} x}\bigg)\bigg)\\
&=\frac{x-d^{*}_{k}}{x-\frac{d^*_k}{p'}}\bigg(1+o\bigg(\frac{d^*_k}{x}\bigg)+\mathit{o}\bigg(\frac{1}{\log^{2} x}\bigg)\bigg).
\end{split}
\end{equation}
Since $d_{k}^{*}\gg\frac{x}{\log^{2}x}$. we have
\begin{align}
G_{k}(x,D_{k}^{*})\leq \bigg(1-\frac{d_{k}^{*}}{3x}\bigg)G_{k}\bigg(x,\frac1{p'}*D_{k}^{*}\bigg)<G_{k}\bigg(x,\frac1{p'}*D_{k}^{*}\bigg),
\end{align}
which contradicts the definition of the $k$-tuple PDC. Thus $d^{*}$ must be square-free, which proves part (ii) of Theorem \ref{mainthm1.3}.


\end{document}